\documentclass[12pt,letterpaper]{article}

\usepackage{amssymb,amsfonts,amscd,amsthm}
\usepackage[all,arc]{xy}
\usepackage{enumerate, bbm}
\usepackage{mathrsfs}
\usepackage{tikz}
\usepackage[left=0.9in,top=0.9in,right=0.9in,bottom=0.9in]{geometry}
\usepackage{mathtools}
\usepackage{hyperref}
\usepackage{color}
\usepackage{graphicx}
\usepackage{enumitem} 
\graphicspath{ {TexImages/} }

\newtheorem{thm}{Theorem}[section]
\newtheorem{cor}[thm]{Corollary}
\newtheorem{prop}[thm]{Proposition}
\newtheorem{lem}[thm]{Lemma}
\newtheorem{claim}[thm]{Claim}

\newtheorem{quest}[thm]{Question}

\theoremstyle{definition}
\newtheorem{defn}{Definition}

\hypersetup{
	colorlinks,
	citecolor=blue,
	filecolor=blue,
	linkcolor=blue,
	urlcolor=blue,
	linktocpage
}

\setlist[enumerate]{itemsep=2ex, topsep=2ex} 
\setlist[itemize]{itemsep=2ex, topsep=2ex}

\newcommand{\F}{\mathbb{F}}

\newcommand{\E}{\mathbb{E}}

\newcommand{\al}{\alpha}

\newcommand{\ep}{\varepsilon}

\newcommand{\om}{\omega}
\newcommand{\Om}{\Omega}

\newcommand{\Del}{\Delta}

\renewcommand{\l}{\left}
\renewcommand{\r}{\right}

\newcommand{\half}{\frac{1}{2}}

\newcommand{\sm}{\setminus}
\newcommand{\sub}{\subseteq}

\renewcommand{\c}[1]{\mathcal{#1}}
\renewcommand{\b}[1]{\mathbf{#1}}

\newcommand{\ol}[1]{\overline{#1}}
\newcommand{\tr}[1]{\textrm{#1}}
\newcommand{\rec}[1]{\frac{1}{#1}}
\newcommand{\f}[2]{\frac{#1}{#2}}
\newcommand{\floor}[1]{\l\lfloor #1\r\rfloor}

\newcommand{\mr}[1]{\mathrm{#1}}

\newcommand{\ex}{\mr{ex}}

\title{Random Polynomial Graphs for Random Tur\'an Problems}
\author{Sam Spiro\footnote{Dept.\ of Mathematics, UCSD {\tt sspiro@ucsd.edu}. This material is based upon work supported by the National Science Foundation Graduate Research Fellowship under Grant No. DGE-1650112.}}

\date{\today}
\begin{document}
	\maketitle
\begin{abstract}
	Bukh and Conlon used random polynomial graphs to give effective lower bounds on $\mathrm{ex}(n,\mathcal{T}^\ell)$, where $\mathcal{T}^\ell$ is the $\ell$th power of a balanced rooted tree $T$.  We extend their result to give effective lower bounds on $\mathrm{ex}(G_{n,p},\mathcal{T}^\ell)$, which is the maximum number of edges in a $\mathcal{T}^\ell$-free subgraph of the random graph $G_{n,p}$.  Analogous bounds for generalized Tur\'an numbers in random graphs are also proven.
\end{abstract}

\section{Introduction}
Given a family of graphs $\c{F}$, we define the \textit{Tur\'an number} $\ex(n,\c{F})$ to be the maximum number of edges that an $n$-vertex $\c{F}$-free graph can have.  If every $F\in \c{F}$ is not bipartite, then the asymptotic behavior of $\ex(n,\c{F})$ is determined by the Erd\H{o}s-Stone theorem~\cite{erdos1946structure}.  Only sporadic results for $\ex(n,\c{F})$ are known when $\c{F}$ contains a bipartite graph, and in most cases these bounds are not tight.  For example, The K{\H{o}}v{\'a}ri-S\'os-Tur\'an theorem~\cite{kHovari1954problem} implies  $\ex(n,K_{s,t})=O(n^{2-1/s})$, and this bound is only known to be tight when $t$ is sufficiently large in terms of $s$; see for example recent work of Bukh~\cite{bukh2021extremal}.  

The bound $\ex(n,C_{2b})=O(n^{1+1/b})$ was first proven by Bondy and Simonovits~\cite{bondy1974cycles}.  It was shown by Faudree and Simonovits~\cite{faudree1983class} that this same upper bound continues to hold for theta graphs, and it was later shown by Conlon~\cite{conlon2019graphs} that this upper bound is tight for theta graphs which have sufficiently many paths. This lower bound of Conlon used random polynomial graphs, which were first introduced by Bukh~\cite{bukh2015random}.  Towards the rational exponents conjecture, Bukh and Conlon~\cite{bukh2018rational} used random polynomial graphs to give tight bounds on the extremal number of powers of balanced trees; see Theorem~\ref{thm:BukhConlon} below for a precise statement.  A more detailed treatment on Tur\'an numbers can be found in the survey by F\"uredi and Simonovits~\cite{furedi2013history}.

Given a graph $G$ and a family of graphs $\c{F}$, we define the \textit{relative Tur\'an number} $\ex(G,\c{F})$ to be the maximum number of edges that an $\c{F}$-free subgraph of $G$ can have.  There are many results for relative Tur\'an number when $G$ is a hypercube \cite{conlon2010extremal,chung1992subgraphs,thomason2009bounding}, and when $G$ is an arbitrary graph \cite{BriggsCox,FKP,GSTZ,PR,SV-K22,SV-Cycles}.  Another well studied case is when $G$ is a random graph, which will be the main focus of this paper.  To this end, let $G_{n,p}$ be the random $n$-vertex graph obtained by including each possible edge independently and with probability $p$. 

Observe that $\ex(G_{n,1},\c{F})=\ex(K_n,\c{F})=\ex(n,\c{F})$.  Because of this observation, it is natural to ask if the classical bounds on $\ex(n,\c{F})$ mentioned above can be extended to give bounds on $\ex(G_{n,p},\c{F})$ for all $p$.  For example, if every $F\in \c{F}$ is not bipartite, then the behavior of $\ex(G_{n,p},\c{F})$ was essentially determined independently in breakthrough work by Conlon and Gowers~\cite{conlon2016combinatorial} and Schacht~\cite{schacht2016extremal}.  The behavior of $\ex(G_{n,p},C_4)$ was essentially determined by F\"uredi~\cite{F}.  This work was substantially generalized by Morris and Saxton~\cite{morris2016number} who proved effective upper bounds for $\ex(G_{n,p},F)$ when $F$ is either an even cycle or a complete bipartite graph.  Very recently, Jiang and Longbrake~\cite{jiang2022balanced} proved general upper bounds for $\ex(G_{n,p},F)$ whenever $F$ satisfies some mild conditions.  As far as we are aware, these are the only known results concerning the random Tur\'an problem for graphs, though there have been a number of recent results regarding the analogous problem for hypergraphs, see for example \cite{MY,NSV,SV-Girth,SV-K22}.

In the same spirit as the work above, we aim to prove a probabilistic analog of the result of Bukh and Conlon~\cite{bukh2018rational} for powers of trees, and for this we need a few definitions. Given a tree $T$ and a set $R\subsetneq V(T)$, we say that the pair $(T,R)$ is a \textit{rooted tree} and refer to $R$ as its set of \textit{roots}.  Given an integer $\ell\ge 1$ and a rooted tree $(T,R)$, we define $\c{T}_R^\ell$ to be the set of graphs consisting of all possible unions of $\ell$ distinct labeled copies of $T$ such that all of these copies agree on the set of roots $R$.  For example, if $T=K_{1,s}$ and $R$ is its set of leaves, then $\c{T}^\ell=\{K_{\ell,s}\}$.    

Given a tree $T$ and a set of vertices $S\sub V(T)$, we define $e_S(T)$ to be the number edges of $T$ which are incident to a vertex of $S$, and we define the \textit{rooted density}\footnote{This is slightly different than the definition used by \cite{bukh2018rational} but will be more convenient for our purposes.} of $(T,R)$ by \[\rho(T,R)=\min_{S\sub V(T)\sm R}\f{e_S(T)}{|S|}.\]  We will often denote $\c{T}_R^\ell$ and $\rho(T,R)$ simply by $\c{T}^\ell$ and $\rho(T)$ whenever $R$ is understood, and similarly we write $e_S$ instead of $e_S(T)$ whenever $T$ is understood.  One of the main results of Bukh and Conlon~\cite[Lemma 1.2]{bukh2018rational} is (essentially\footnote{Strictly speaking they only state their theorem for ``balanced'' trees, i.e.\ those with $\rho(T)=\frac{e(T)}{|V(T)|-|R|}$.  Theorem~\ref{thm:BukhConlon} can be deduced from an identical proof.  Alternatively, it can be obtained by applying \cite[Lemma 1.2]{bukh2018rational} to the balanced rooted tree $(T,R')$ where $R'=R\cup (V(T)\sm S)$ and $S$ is the set achieving $\rho(T,R)$.}) the following.
\begin{thm}[\cite{bukh2018rational}]\label{thm:BukhConlon}
	Let $(T,R)$ be a rooted tree with $\rho(T)\ge \f{b}{a}$ for $a,b$ positive integers. 
	Then there exists some $\ell_0$ depending only on $(T,R)$ such that, for all $\ell\ge \ell_0$, 
	\[\ex(n,\c{T}^\ell)=\Om\l(n^{2-\f{a}{b}}\r).\]
\end{thm}

Our main theorem is a probabilistic analog of this result.  For this statement, we recall that a sequence of events $A_n$ is said to hold \textit{asymptotically almost surely}, or a.a.s.\ for short, if $\Pr[A_n]$ tends to 1 as $n$ tends towards infinity.
\begin{thm}\label{thm:main}
	Let $(T,R)$ be a rooted tree with $\rho(T)\ge \f{b}{a}$ for $a,b$ positive integers. Then there exists some $\ell_0$ depending only on $(T,R)$ such that, for all $\ell\ge \ell_0$, we have for all $p=p(n)$ with $pn \log n\to \infty$ that a.a.s.
	\[\ex(G_{n,p},\c{T}^\ell)=\Omega\left(p^{1-\f{a}{b}}n^{2-\f{a}{b}}\right).\]
\end{thm}

\textbf{Note}: the hypothesis $pn\log n\to \infty$ here and in Theorem~\ref{thm:mainGeneralized} below corrects the moderately stronger hypothesis $pn^2\to \infty$ which was (incorrectly) claimed in the journal version of this article.

Theorem~\ref{thm:main} is known to be sharp in certain cases when $p$ is sufficiently large.  For example, work of Morris and Saxton~\cite{morris2016number} shows that Theorem~\ref{thm:main} gives the correct lower bound on $\ex(G_{n,p},K_{s,t})$ provided $t$ is sufficiently large in terms of $s$ whenever $p\gg n^{-(s-1)/(st-1)}$; see the concluding remarks in Section~\ref{sec:con} for more on this.

While Theorem~\ref{thm:main} is the main goal of this paper, essentially the same argument will allow us to extend this result to generalized Tur\'an numbers.  Given two graphs $G,H$, we define $\c{N}_H(G)$ to be the number of copies of $H$ in $G$, i.e.\ the number of subgraphs $K\sub G$ which are isomorphic to $H$.  Given a family of graphs $\c{F}$, we define the \textit{generalized relative Tur\'an number}
\[\ex(G,H,\c{F})=\max\{\c{N}_H(G'):G'\sub G,\ G' \tr{ is }\c{F}\tr{-free}\}.\]
For example, $\ex(G,K_2,\c{F})$ is the maximum number of edges that an $\c{F}$-free subgraph of $G$ can have, which is just the relative Tur\'an number $\ex(G,\c{F})$.  When $G=K_n$ we write $\ex(n,H,\c{F}):=\ex(K_n,H,\c{F})$ and call this the \textit{generalized Tur\'an number}, which is the maximum number of copies of $H$ that an $\c{F}$-free $n$-vertex graph can contain.  Generalized Tur\'an numbers were first systematically studied by Alon and Shikhelman~\cite{AS}, and since then many results have been proven, see for example~\cite{gerbner2019counting,gishboliner2020generalized,gyHori2018maximum}.  

In this work we study the function $\ex(G,H,\c{F})$ when the host graph $G$ is a random graph.  Some work  in this direction has been done by Alon, Kostochka, and Shikhelman~\cite{alon2016many} and by Samotij and Shikhelman~\cite{samotij2020generalized}.  These papers only consider the case when the graphs $\c{F}$ are non-bipartite, and as far as we are aware this is the first paper to consider the bipartite case of this problem.  To state our result, we define the \textit{density} of a graph $H$ by
\[m(H)=\max\left\{\f{e(H')}{v(H')}:H'\sub H\right\},\]
where here and throughout $v(H),e(H)$ denote the number of vertices and edges of $H$.
With this we have the following, which immediately implies Theorem~\ref{thm:main} by taking $H=K_2$.
\begin{thm}\label{thm:mainGeneralized}
	Let $H$ be a graph with $v$ vertices and $e$ edges, and let $(T,R)$ be a rooted tree with $\rho(T)\ge \f{b}{a}$ for $a,b$ positive integers. Then there exists some $\ell_0$ depending only on $(T,R)$ and $H$ such that, for all $\ell\ge \ell_0$, we have for all $p=p(n)$ with $p^{m(H)}n\to \infty$ and $pn\log n\to \infty$ that a.a.s.
	\[\ex(G_{n,p},H,\c{T}^\ell)=\Om\left(p^{e-\f{a}{b}\cdot e}n^{v-\f{a}{b}\cdot e}\right).\]
\end{thm}
We note that it is easy to show by a first moment argument that $\c{N}_H(G_{n,p})=0$ with high probability if $p^{m(H)}n\to 0$, so it is reasonable to only consider $p$ with $p^{m(H)}n\to \infty$.

\section{Key Propositions}
In this section we state without proof the tools we need in order to prove Theorem~\ref{thm:mainGeneralized}.  These results are somewhat general, and we have hope that they will find applications outside of the present work.

We begin with a generalization of Theorem~\ref{thm:BukhConlon}.  To state this result; given a graph $F$ and a set $R\subsetneq V(F)$, we will say that the pair $(F,R)$ is a \textit{rooted graph} and refer to $R$ as its set of \textit{roots}.  We define $\c{F}^\ell$ and $\rho(F)$ exactly analogous to how these were defined for rooted trees. 

\begin{prop}\label{prop:randPolyGen}
	Let $H$ be a graph with $v$ vertices and $e$ edges, and let $\{(F_1,R_1),\ldots,(F_t,R_t)\}$ be a set of rooted graphs such that $\rho(F_i)\ge \f{b}{a}$ for all $i$ with $a,b$ positive integers.  Then there exists some $\ell_0$ depending only on $H$ and $\{(F_1,R_1),\ldots,(F_t,R_t)\}$ such that, for all $\ell\ge \ell_0$, 
	\[\ex(n,H,\bigcup \c{F}_i^\ell)=\Om\l(n^{v-\f{a}{b}\cdot e}\r).\]
\end{prop}

We emphasize that it is straightforward to derive Proposition~\ref{prop:randPolyGen} by using the approach of Ma, Yuan, and Zhang~\cite{ma2018some}, which itself is based heavily off of the random polynomial approach of Bukh and Conlon~\cite{bukh2018rational}.  We include the details for completeness in Section~\ref{sec:randPoly}.

Our lasts two results are new and utilize the concept of local isomorphisms.  For this definition, we adopt the convention $\phi(S):=\{\phi(x):x\in S\}$ whenever $\phi$ is a map and $S$ is a set.

\begin{defn}
	Given two rooted graphs $(F,R),(F',R')$, we say that a map $\phi:V(F)\to V(F')$ is a \textit{local isomorphism} if
	\begin{itemize}
		\item[(a)] The map $\phi$ is surjective with $\phi(R)\sub R'$,
		\item[(b)] Every $e\in E(F)$ has $\phi(e)\in E(F')$ (i.e.\ $\phi$ is a homomorphism), and
		\item[(c)] Every $e,f\in E(F)$ with $e\ne f$ and $e\cap f\ne \emptyset$ satisfy $\phi(e)\ne \phi(f)$.
	\end{itemize}
	Given a rooted graph $(F,R)$, we let $\c{L}(F)$ denote the set of rooted graphs $(F',R')$ such that there exists a local isomorphism $\phi:V(F)\to V(F')$, and given a set of rooted graphs $\c{F}$, we define $\c{L}(\c{F})=\bigcup_{F\in \c{F}} \c{L}(F)$. 
\end{defn}
Our main motivation for these definitions is the following result for graphs without roots, which allows us to transfer lower bounds for $\ex(m,H,\c{L}(\c{F}))$ to lower bounds for $\ex(G,H,\c{F})$.

\begin{prop}\label{prop:relTuran}
	Let $G,H$ be graphs, and let $\c{F}$ be a family of graphs.  If $m$ is an integer satisfying $m\ge 4|E(H)| \Del$ where $\Del$ is the maximum degree of $G$, then there exists a $G'\sub G$ which is $\c{F}$-free such that
	\[\c{N}_H(G')\ge \half \ex(m,H, \c{L}(\c{F})) m^{-v(H)}\cdot \c{N}_H(G).\]
\end{prop}

We anticipate that Proposition~\ref{prop:relTuran} will be quite useful in obtaining lower bounds for generalized relative Tur\'an numbers.  With this result in hand, we see that to prove Theorem~\ref{thm:mainGeneralized}, it suffices to prove effective lower bounds on $\ex(n,H,\c{L}(\c{T}^\ell))$.  We will do this by utilizing Proposition~\ref{prop:randPolyGen} together with the following result, which says  that the rooted density of a forest can only increase after taking local isomorphisms. 
\begin{prop}\label{prop:treeDensity}
	If $(T,R_T)$ is a rooted forest and $(F,R_F)\in \c{L}(T)$, then $\rho(T)\le \rho(F)$.
\end{prop}
We will only apply Proposition~\ref{prop:treeDensity} when $T$ is a tree, but its proof is somewhat more natural in the setting of forests. 

\section{Proof of Theorem~\ref{thm:mainGeneralized}}
We postpone proving the results of the previous section for the moment, and instead show how these propositions quickly imply our main result.  We begin with a deterministic result 
\begin{lem}\label{lem:localIsoTuran}
	Let $H$ be a graph with $v$ vertices and $e$ edges, and let $(T,R_T)$ be a rooted tree such that $\rho(T)\ge \f{b}{a}$ with $a,b$ positive integers.  Then there exists some $\ell_0$ depending only on $H$ and $(T,R_T)$ such that, for all $\ell\ge \ell_0$, 
	\[\ex(n,H,\c{L}(\c{T}^\ell))=\Om\l(n^{v-\f{a}{b}\cdot e}\r).\]
\end{lem}
For this proof and throughout the paper, given a graph $G$ and a set of vertices $S$, we define $G[S]$ to be the induced subgraph of $G$ with vertex set $S$.
\begin{proof}
	Let $c$ denote the number of rooted graphs that are in $\c{L}(T)$ up to isomorphism, and observe that $c<\infty$ because local isomorphisms are in particular surjective homomorphisms.  Let $\ell_*:=\floor{\ell/c}$, noting that we can take $\ell_*$ to be arbitrarily large by taking $\ell$ to be arbitrarily large.
	
	We claim that for every $F\in \c{L}(\c{T}^\ell)$, there exists some rooted graph $F_*\in \c{L}(T)$ such that $F$ contains an element of $\c{F}_*^{\ell_*}$ as a subgraph.  Assuming this claim, we see that for a graph to be $\c{L}(\c{T}^\ell)$-free, it suffices for it to be $\bigcup_{F_*\in  \c{L}(T)}\c{F}_*^{\ell_*}$-free.  Note that $\rho(F_*)\ge \f{b}{a}$ for all $F_*\in  \c{L}(T)$ by Proposition~\ref{prop:treeDensity}, so by Proposition~\ref{prop:randPolyGen} we have for $\ell$ sufficiently large that
	\[\ex(n,H,\c{L}(\c{T}^\ell))\ge \ex(n,H,\bigcup_{F_*\in \c{L}(T)} \c{F}_*^{\ell_*})=\Om\left(n^{v-\f{a}{b}\cdot e}\right).\]
	It remains to prove this claim.
	
	By definition of the set $\c{T}^\ell$, if $F\in \mathcal{L}(\mathcal{T}^\ell)$, then there exists a local isomorphism $\phi:V(\bigcup_{i=1}^{\ell} T_i)\to V(F)$ where $T_1,\ldots,T_\ell$ are $\ell$ distinct copies of $T$ which agree on $R_T$.  Define the induced subgraph $F_i=F[\phi(V(T_i))]$ for all $i$.  Note that the restriction $\phi:V(T_i)\to V(F_i)$ is a local isomorphism from $(T_i,R_T)$ to $(F_i,\phi(R))$ because $\phi$ is a local isomorphism.  By the pigeonhole principle and the definition of $\ell_*$, there exists some $(F_*,R_*)\in \c{L}(T)$ such that at least $\ell_*$ different integers $i$ have $(F_i,\phi(R))$ isomorphic to $(F_*,R_*)$. Let $I$ denote the set of these $i$.  By definition, $\bigcup_{i\in I} F_i$ is an element of $\c{F}_*^{\ell_*}$ and also a subgraph of $F$, so we conclude the claim and hence the result.
\end{proof}

We can now prove our main result.
\begin{proof}[Proof of Theorem~\ref{thm:mainGeneralized}]
	Let $A$ denote the event that $\c{N}_H(G_{n,p})\ge \half \E[\c{N}_H(G_{n,p})]=\Theta(p^e n^v)$ and that $G$ has maximum degree at most $2pn$. Using, for example, a result of Janson~\cite{janson1990poisson} and the Chernoff bound, one can easily show $A$ holds a.a.s.\ provided $p^{m(H)}n\to \infty$ and $pn\log n\to \infty$.  Conditional on $A$ occurring, we can apply  Proposition~\ref{prop:relTuran} with some appropriate $m=\Theta(pn)$ together with Lemma \ref{lem:localIsoTuran} to conclude that, for $\ell$ sufficiently large,
	\[\ex(G_{n,p},H,\c{F})\ge \half \ex(m,H,\c{L}(\c{F})) m^{-v}\cdot \c{N}_H(G_{n,p})=\Om\left(m^{-\f{a}{b}\cdot e}\cdot p^e n^v\right)=\Om\left(p^{e-\f{a}{b}\cdot e}n^{v-\f{a}{b}\cdot e}\right),\]
	giving the result.
\end{proof}
It remains to prove Propositions~\ref{prop:randPolyGen}, \ref{prop:relTuran}, and \ref{prop:treeDensity}; which we do in the following two sections.
\section{Proof of Proposition~\ref{prop:randPolyGen}}\label{sec:randPoly}
We again emphasize that our proof of Proposition~\ref{prop:randPolyGen} is nearly identical to that of Ma, Yuan, and Zhang~\cite{ma2018some} and of Bukh and Conlon~\cite{bukh2018rational}.
We begin with a simple combinatorial observation.

\begin{lem}\label{lem:edgeBound}
	For every rooted graph $(F,R)$ and integer $\ell\ge 1$, we have that every $K\in \c{F}^\ell$ satisfies
	\[e(K)\ge \rho(F)(v(K)-|R|)\]
\end{lem}
\begin{proof}
	We prove the result by induction on $\ell$.  If $\ell=1$ then $K=F$, so the claimed inequality is equivalent to having \[ \f{e(F)}{v(F)-|R|}\ge \rho(F):=\min_{S\sub V(F)\sm R}\f{e_S}{|S|},\] 
	where we recall that $e_S$ is the number of edges of $F$ incident to a vertex of $S$.  This inequality is trivially true, proving the base case.
	
	Assume now that we proven the result up to some value $\ell>1$, and let $K\in \c{F}^\ell$.  This means $K$ can be written as the union of $\ell$ labeled copies of $F$, say $F_1,\ldots,F_\ell$, such that each of these agree on the set of roots $R$.  Let $K':=\bigcup_{i<\ell} F_i$ and let $S$ be the set of vertices in $F_\ell$ which are not in $K'$, noting that necessarily $S\sub V(F)\sm R$.  By definition of $\rho(F)$, we must have $e_S\ge \rho(F)|S|$.  Using this and our inductive hypothesis gives
	\[e(K)= e(K')+e_S\ge \rho(F)(v(K')-|R|)+\rho(F)|S|=\rho(F)(v(K)-|R|),\]
	giving the result. 
\end{proof}

We now introduce some algebraic definitions and lemmas.  Let $\F_q$ denote the finite field with $q$ elements. We write $\F_q[\b{X}^1,\b{X}^{2}]$ with $\b{X}^j=(X^j_1,\ldots,X_b^j)$ to denote the set of polynomials with variables $X_i^j$ which have coefficients in $\F_q$.  Given such a polynomial $f$ and $x^1,x^2\in \F_q^b$, we define $f(x^1,x^2)\in \F_q$, by replacing each $X_i^j$ symbol in $f$ with $x_i^j\in \F_q$. 

We say that a polynomial $f\in \F_q[\b{X}^1,\b{X}^{2}]$ has degree at most $d$ in $\b{X}^j$ if each of its monomials with respect to $\b{X}^j$, say $(X^j_1)^{\al_1}\cdots (X_b^j)^{\al_b}$, satisfies $\sum_{i=1}^b \al_i\le d$.  We say that such a polynomial is \textit{pseudo-symmetric} if $f(x^1,x^2)=f(x^2,x^1)$ for all $x^1,x^2\in \F_q^{b}$.  Let $\c{P}_{d,2b}\sub  \F_q[\b{X}^1,\b{X}^{2}]$ denote the set of all pseudo-symmetric polynomials in $2b$ variables which have degree at most $d$ in $\b{X}^j$ for all $j$.  

\begin{lem}[\cite{ma2018some} Lemma 2.2]\label{lem:symmetric}
	Let $V\sub \F_q^b$ and $E\sub {V\choose 2}$ be such that ${|V|\choose 2},{|E|\choose 2}<q$.  If $|E|\le d$, and if $f$ is a random polynomial chosen uniformly from $\c{P}_{d,2b}$, then
	\[\Pr[f(x^1,x^2)=0\ \ \forall \{x^1,x^2\}\in E]=q^{-|E|}.\]
\end{lem}

We next need some definitions from algebraic geometry.  If $\ol{\F}_q$ denotes the algebraic closure of $\F_q$, then a \textit{variety} over $\ol{\F}_q$ is any set of the form $X=\{x\in \ol{\F}_q^b:f_1(x)=\cdots=f_a(x)=0\}$ where $f_1,\ldots,f_a:\ol{\F}_q^b\to \F_q$ are polynomials.  We say that $X$ is defined over $\F_q$ if the coefficients of its polynomials are in $\F_q$, and in this case we write $X(\F_q)=X\cap \F_q^b$. The variety $X$ is said to have complexity at most $M$ if $a,b$ and the degrees of the polynomials $f_k$ are bounded by $M$.  One can prove the following using standard results from algebraic geometry.

\begin{lem}[\cite{bukh2018rational} Lemma 2.7]\label{lem:varieties}
	Let $X,D$ be varieties over $\ol{\F}_q$ of complexity at most $M$ which are defined over $\F_q$.  If $q$ is sufficiently large in terms of $M$, then either $|X(\F_q)\sm D(\F_q)|\ge q/2$ or $|X(\F_q)\sm D(\F_q)|\le c$ for some $c$ depending only on $M$.
\end{lem}

We can now prove our main result for this section.

\begin{proof}[Proof of Proposition~\ref{prop:randPolyGen}]
	Recall that we wish to find an $\bigcup \c{F}_i^\ell$-free graph which has many copies of a graph $H$ that has $v$ vertices and $e$ edges.  Define \[s=1+ae+b\max_i(|R_i|-1),\hspace{.8em} d=s\max_i e(F_i),\hspace{.8em} N=q^b,\] where $q$ will be any prime power which is sufficiently large in terms of $H$, the pairs $(F_i,R_i)$, and the $\ell_0$ which we choose later on.  Let $f_1,\ldots,f_a\in \c{P}_{d,2b}$ be chosen independently and uniformly at random.  Let $G$ be the (random) graph with vertex set $\F_q^b$ where distinct vertices $x^1,x^2\in \F_q^b$ form an edge of $G$ if and only if $f_k(x^1,x^2)=0$ for all $k=1,\ldots,a$.  Note that the $f_k$ being pseudo-symmetric implies that the set of edges is well defined. 
	
	By Lemma~\ref{lem:symmetric}, for $q$ sufficiently large, the probability that a given set of $v$ vertices of $G$ forms a copy of $H$ is at least $q^{-ae}$.  Thus  \begin{equation}\E[\c{N}_H(G)]\ge q^{-ae} {N\choose v}=\Om(N^{v-\f{a}{b} e}).\label{eq:NHG}\end{equation}
	
	Consider some $(F_i,R_i)$, and let $u_1,\ldots,u_r$ denote its set of roots.  Fix vertices $x^1,\ldots,x^r\in \F_q^b$, and let $C$ be the collection of copies of $F_i$ in $G$ such that $x^j$ corresponds to $u_j$ for all $j$.  Our goal will be to estimate the $s$th moment of $|C|$.  To aid slightly with our proof, we let $\tilde{G}$ denote the complete graph with vertex set $V(G)=\F_q^b$.
	
	By definition, $|C|^s$ counts the number of ordered collections of $s$ (possibly overlapping, and possibly identical) copies of $F_i$ in $G$ such that $x^j\in V(G)$ corresponds to $u_j$ for all $j$.  Given a graph $K$, define $\chi_s(K)$ to be the number of such ordered collections of $s$ copies of $F_i$ in $\tilde{G}$ such that the union of the elements in the collections is isomorphic to $K$.  By definition, $\chi_s(K)=0$ for all $K\notin \bigcup_{\ell=1}^s \c{F}_i^\ell:=\c{F}_i^{\le s}$, and for all other $K$ we have \[\chi_s(K)=O_s(N^{v(K)-|R_i|})=O_s(q^{b(v(K)-|R_i|)}).\]  Since $e(K)\le d$ for all $K\in \c{F}_i^{\le s}$ by definition of $d$, we have by Lemma~\ref{lem:symmetric} for $q$ sufficiently large that the probability a given copy of $K$ lies in $G$ is exactly $q^{-a\cdot e(K)}$.  Putting all this together, we find
	\begin{equation}\E[|C|^s]=\sum_{K\in \c{F}_i^{\le s}} \chi_s(K) q^{-a\cdot e(K)}=\sum_{K\in \c{F}_i^{\le s}} O_s\l(q^{b(v(K)-|R_i|)-a\cdot e(K)}\r)=O_s(1),\label{eq:C}\end{equation}
	where this last step used Lemma~\ref{lem:edgeBound} together with the assumption $\rho(F_i)\ge \f{b}{a}$ and that $K\in \c{F}_i^\ell$ for some $\ell\le i$.

	Write $V(F_i)=\{u_1,\ldots,u_{r+p}\}$, and let $X$ be the algebraic variety consisting of all the tuples $(x^{r+1},\ldots,x^{r+p})\in \ol{\F}_q^{pb}$ such that for all $k=1,\ldots,a$, we have $f_k(x^{j_1},x^{j_2})=0$ whenever $u_{j_1}u_{j_2}\in E(F_i)$.  Observe that $C\sub X(\F_q)$, but the two sets are not equal due to the presence of degenerate copies of $F_i$ counted by $X$.  Namely, for all $j\ne j'$, define
	\[D_{j,j'}=X\cap \{(x^{r+1},\ldots,x^{r+p}):x^j=x^{j'}\},\]
	and set $D=\bigcup_{j\ne j'} D_{j,j'}$.  With this we have $C=X(\F_q)\sm D(\F_q)$.  Because $D$ is the union of varieties of bounded complexity, it too has bounded complexity.  Thus by Lemma~\ref{lem:varieties}, there exists some constant $c_i$ depending only on $F_i$ such that either $|C|\le c_i$ or $|C|\ge q/2$.  Using this together with  Markov's inequality and \eqref{eq:C}, we find 
	\[\Pr[|C|> c_i]=\Pr[|C|\ge q/2]=\Pr[|C|^s\ge (q/2)^s]\le \f{\E[|C|^s]}{(q/2)^s}=O_s(q^{-s}).\]
	Call a sequence $(x^1,\ldots,x^r)$\textit{ $i$-bad} if there are more than $c_i$ copies of $F_i$ in $G$ such that $x^j$ corresponds to $u_j\in R$ for all $1\le j\le r$.  Let $B_i$ denote the number of $i$-bad sequences.  Our analysis above gives
	\begin{equation}\E[B_i]\le N^r\cdot O_s(q^{-s})=O_s(N^{r-s/b})=o(N^{1-\f{a}{b}e}),\label{eq:Bi}\end{equation}
	where this last step used $s\ge 1+ae+b(r-1)$.
	
	Now let $G'\sub G$ be defined by deleting a vertex from each $i$-bad sequence for all $i$.  Because each vertex is in at most $vN^{v-1}$ copies of $H$ in $G$, by \eqref{eq:NHG} and \eqref{eq:Bi} we find
	\[\E[\c{N}_H(G')]\ge \E[\c{N}_H(G)]-\E\left[\sum_i B_i\right]\cdot vN^{v-1}=\Om(N^{v-\f{a}{b}e}).\]
	By definition $G'$ contains no element of $\bigcup \c{F}_i^{c_i}$, so for $\ell\ge \ell_0:=\max c_i$, we have shown that there exists a graph $G'$ on at most $N$ vertices such that it contains at least $\Om(N^{v-\f{a}{b}e})$ copies of $H$ and contains no element of $\bigcup \c{F}_i^{\ell}$.  This gives the desired lower bound on $\ex(n,H,\bigcup \c{F}_i^{\ell})$ when $n=q^b$ and $q$ is a sufficiently large prime power.  Using Bertrand's postulate gives the desired bound for all $n$.
\end{proof}

\section{Proof of Propositions~\ref{prop:relTuran} and \ref{prop:treeDensity}}
In this section we prove our results regarding local isomorphisms, which we recall are surjective homomorphisms  $\phi:V(F)\to V(F')$ between rooted graphs with $\phi(R)\sub R'$ such that $e\cap f\ne \emptyset$ implies $\phi(e)\ne \phi(f)$ whenever $e,f$ are distinct edges of $F$.  We also recall that $\c{L}(\c{F})$ consists of all of the graphs $F'$ for which there exists a graph $F\in \c{F}$ which has a local isomorphism to $F'$.

We begin by proving our ``transference'' result, which we recall says that if $G,H$ are graphs and $m$ is an integer with $m\ge 4|E(H)| \Del$, then there exists an $\c{F}$-free subgraph $G'\sub G$ with
\[\c{N}_H(G')\ge \half \ex(m,H, \c{L}(\c{F})) m^{-v(H)}\cdot \c{N}_H(G).\]

\begin{proof}[Proof of Proposition~\ref{prop:relTuran}]
	Let $M$ be an $m$-vertex graph with $\c{N}_H(M)=\ex(m, H,\c{L}(\c{F}))$ and which is $\c{L}(\c{F})$-free (which exists by definition), and let $\phi:V(G)\to V(M)$ be chosen uniformly at random amongst all possible maps from $V(G)$ to $V(M)$.  Let $G'\sub G$ be the subgraph obtained by keeping all of the edges $e\in E(G)$ which satisfy
	\begin{itemize}
		\item[$(1)$] $\phi(e)\in E(M)$, and
		\item[$(2)$] $\phi(e)\ne \phi(f)$ for any other $f\in E(G)$ with $e\cap f\ne \emptyset$.  
	\end{itemize}
	We claim that $G'$ is $\c{F}$-free.  Indeed, assume $G'$ contained a subgraph $F$ isomorphic to some element of $\c{F}$.  Let $F'$ be the subgraph of $M$ with $V(F')=\{\phi(u):u\in V(F)\}$ and $E(F')=\{\phi(e):e\in E(F)\}$.  Note that $F\sub G'$ implies that each edge of $F$ satisfies (1), so every element of $E(F')$ is an edge in $M$, i.e.\ $F'$ is indeed a subgraph of $M$.  By conditions (1) and (2), $\phi$ is a local isomorphism from $F$ to $F'$, so $F'\in \c{L}(\c{F})$, a contradiction to $M$ being $\c{L}(\c{F})$-free.
	
	It remains to compute $\E[\c{N}_H(G')]$.  Fix a subgraph $K\sub G$ which is isomorphic to $H$.  Let $A$ denote the event that $\phi$ maps $V(K)$ isomorphically onto a copy of $H$ in $M$ (which in particular means (1) is satisfied for each $e\in E(K)$), and let $B$ denote the event that (2) is satisfied for every $e\in E(K)$.  It is not difficult\footnote{Somewhat more formally, by definition there exists an isomorphism $\psi:V(K)\to V(H)$. If $L\sub M$ is a copy of $H$, then by definition there exists an isomorphism $\psi_L:V(L)\to V(H)$. Each of the $\c{N}_H(M)$ maps $\psi_L^{-1}\circ \psi:V(K)\to V(L)\sub V(M)$ are distinct, and $\phi$ satisfies $A$ if and only if $\phi$ restricted to $V(K)$ is one of these maps.} to see that  $\Pr[A]= \c{N}_H(M) m^{-v(H)}$.  To bound $\Pr[B|A]$, we claim that if $|e\cap f|=1$, then 
	\begin{equation}
		\Pr[\phi(e)= \phi(f)|A]\le \f{1}{m}.\label{eq:B}
	\end{equation}
	Indeed, if $f\sub V(K)$, then the event $\phi(e)= \phi(f)$ can never occur given $A$ (since in particular, $A$ implies that $\phi$ restricted to $V(K)$ is a bijection), and otherwise we have $\Pr[\phi(e)= \phi(f)|A]= \f{1}{m}$.  Note that $B$ occurs provided each $e\in E(K)$ has $\phi(e)\ne \phi(f)$ for all $f$ with $|e\cap f|=1$.   Thus a union bound together with \eqref{eq:B} gives
	\begin{align*}\Pr[B|A]&\ge1- |E(H)|\cdot 2\Del\cdot (1/m)\ge \half. \end{align*}
	
	In total we find
	\[\Pr[E(K)\sub E(G')]\ge \Pr[A\cap B]=\Pr[A]\cdot \Pr[B|A]\ge  \c{N}_H(M) m^{-v(H)}\cdot \half,\]
	and linearity of expectation gives \[\E[\c{N}_H(G')]\ge \half  \c{N}_H(M) m^{-v(H)}\cdot \c{N}_H(G)=\half \ex(m,H,\c{L}(\c{F}))m^{-v(H)}\cdot \c{N}_H(G).\]  Thus there exists some subgraph $G'\sub G$ with this many copies of $H$ which is $\c{F}$-free, giving the result
\end{proof}

The remainder of the section is dedicated to proving  Proposition~\ref{prop:treeDensity}, which we recall says that if $T$ is a rooted forest and $\phi:V(T)\to V(F)$ is a local isomorphism, then $\rho(T)\le \rho(F)$.  Having $\rho(T)\le \rho(F)$ is equivalent to saying that for any $X\sub V(F)\sm R_F$, there exists a set $V\sub V(T)\sm R_T$ such that $\f{e_V(T)}{|V|}\le \f{e_X(F)}{|X|}$.  To prove this, we will show that for each $X\sub V(F)\sm R_F$, there exists some $V\sub V(T)\sm R_T$ such that $\phi$ maps $V$ bijecitvely onto $X$ and such that $e_V(T)\le e_X(F)$.

With this goal in mind, we define $E_S(G)$ to be the set of edges of a graph $G$ which contain a vertex of $S$, and we denote this simply be $E_S$ whenever $G$ is understood.  Note that $|E_S|=e_S$.  Given a local isomorphism $\phi:V(T)\to V(F)$, we define the induced map $\phi^*:E(T)\to E(F)$ by mapping $uv\in E(T)$ to $\phi(u)\phi(v)\in E(F)$.  We say that a set of vertices $V\sub V(T)\sm R_T$ is \textit{good} (with respect to $\phi$) if $\phi$ maps $V$ bijectively onto $V(F)$ and if $\phi^*$ restricted to $E_V$ is an injection.

We first show that good sets exist provided $T$ is a forest without roots.
\begin{lem}\label{lem:tech}
	If $T$ is a forest without roots and 	$\phi:V(T)\to V(F)$ is a local isomorphism, then there exists a good set $V$ with respect to $\phi$.
\end{lem}
\begin{proof}
	Assume for contradiction that this does not hold for some $T,F$, and choose such a pair with $v(T)$ as small as possible.  Note that the result trivially holds if $T=\emptyset$, so we can assume $T$ has some $u\in V(T)$, and we can choose this $u$ to have degree at most 1 since $T$ is a forest.  
	\begin{claim}
		The map $\phi$ restricted to $T-u$ is a local isomorphism.
	\end{claim}
	\begin{proof}
		The only way $\phi$ restricted to the smaller graph $T-u$ can fail to be a local isomorphism is if it is no longer surjective, i.e.\ if $\phi(w)\ne \phi(u)$ for any $w\ne u$.  In this case, $\phi$ restricted to the forest $T-u$ is a local isomorphism onto $F-\phi(u)$, and by assumption of $T$ being a minimal counterexample, there exists some good set $V'$ for this map.  We claim that $V=V'\cup \{u\}$ is a good set for $\phi:V(T)\to V(F)$.  
		
		By assumption $\phi$ maps $V'$ bijectively onto $V(F)\sm \{\phi(u)\}$, so $\phi$ also maps $V$ bijectively onto $V(F)$.  We also have that $\phi^*$ is injective on $E_{V'}$.  Note that $E_V=E_{V'}$ if $u$ has degree 0, in which case $\phi^*$ is injective on $E_V$, and otherwise $E_V=E_{V'}\cup \{uv\}$ where $v$ is the unique neighbor of $u$.  Thus if $\phi^*$ is not injective on $E_V$, we must have $\phi^*(xy)=\phi^*(uv)$ for some $xy\ne uv$.  This implies, say, $\phi(x)=\phi(u)$, which means $x=u$ since $\phi(w)\ne \phi(u)$ for any $w\ne u$.  This in turn means $y=v$ since $v$ is the unique neighbor of $u=x$, contradicting our assumption $xy\ne uv$.  We conclude that $\phi^*$ is injective on $E_V$, and hence $V$ is good.  This contradicts our choice of $T$, proving the claim.

	\end{proof}
	
	By the assumption that $T$ is a minimal counterexample, there must exist a good set $V'$ with respect to the local isomorphism $\phi:V(T-u)\to V(F)$.  If $u$ is not adjacent to any vertex in $V'$, then we claim that $V'$ is a good set with respect to $\phi:V(T)\to V(F)$.  Indeed, by assumption $\phi$ maps $V'$ bijectively onto $V(F)$, and we have that $\phi^*$ is injective on $E_{V'}(T-u)=E_V(T)$.  This contradicts $T$ being a counterexample. 
	
	Thus there exists some $v\in V'$ which is adjacent to $u$, and we note that this must is the unique neighbor of $u$ since we assumed $u$ has degree at most 1.  Because $V'$ is mapped bijectively onto $V(F)$, there exists some $u'\in V'$ with $\phi(u')=\phi(u)$.  Define $V=V'\sm \{u'\}\cup \{u\}$.  We claim that this set is good.  It is clear that $\phi$ maps $V$ bijectively onto $V(F)$, so it remains to show that $\phi^*$ is injective on $E_V$.
	
	Observe that $E_{V}\sub E_{V'}\cup \{uv\}$, so if $\phi^*$ were not injective, then we must have $\phi^*(xy)=\phi^*(uv)$ for some $xy\ne uv$ in $E_V$.  Assume without loss of generality that $x\in V$.  If $\phi(x)=\phi(u)$ then $x=u$ (since $\phi$ is bijective on $V$) which implies $y=v$, contradicting $xy\ne uv$.  Thus $\phi(x)=\phi(v)$ and hence $x=v$.  This implies that $\phi^*$ maps two distinct edges $uv,vy$ with a vertex in common to the same edge of $F$, a contradiction to $\phi$ being a local isomorphism.  We conclude that $\phi^*$ is injective on $E_V$ and hence $V$ is good.  This contradicts our choice of $T$, proving the result.

\end{proof}
We can now prove our main result.
\begin{proof}[Proof of Proposition~\ref{prop:treeDensity}]
	Let $X\sub V(F)\sm R_F$ be a set of vertices such that $\f{e_X}{|X|}=\rho(F)$, and let $\phi^{-1}(X)=\{u\in V(T):\phi(u)\in X\}$.  Note that $\phi^{-1}(X)$ contains no roots of $T$ because $X\sub V(F)\sm R_F$ and because $\phi$ is a local isomorphism.   Let $F'=F[X]$ and $T'=T[\phi^{-1}(X)]$, noting that $\phi:V(T')\to V(F')$ is still a local isomorphism and that $F',T'$ are unrooted graphs with $T'$ a forest.  By the previous lemma, there exists a set $V\sub \phi^{-1}(X)\sub V(T)\sm R_T$ which is good with respect to the map $\phi:V(T')\to V(F')$.
	
	\begin{claim}
		The map $\phi^*:E_V(T)\to E_X(F)$ is injective.
	\end{claim}
	\begin{proof}
		Assume for contradiction that there were distinct edges with $\phi^*(wx)=\phi^*(yz)$, say with $\phi(w)=\phi(y)$ and $\phi(x)=\phi(z)$.  Note that we cannot have $w,x,y,z\in \phi^{-1}(X)=V(T')$, since $V$ being good implies that $\phi^*$ is injective on such edges.  Thus we can assume without loss of generality that $w\notin \phi^{-1}(X)$, and having $\phi(w)=\phi(y)$ forces $y\notin \phi^{-1}(X)$ as well.  Because $V\sub \phi^{-1}(X)$, we have $w,y\notin V$, so by definition of $wx,yz\in E_V(T)$ we must have $x,z\in V$.  Having $\phi(x)=\phi(z)$ implies $x=z$ since $\phi$ is bijective on $V$.  Thus $w,y$ are distinct neighbors of $x$ with $\phi(w)=\phi(y)$, a contradiction to $\phi$ being a local isomorphism.  This proves the claim.
	\end{proof}
	Note that $|E_V(T)|=e_V$ and $|E_X(F)|=e_X$. Because $|V|=|X|$ by definition of $V$ being good, the claim implies
	\[\rho(T)=\min_{S\sub V(T)\sm R_T}\f{e_{S}}{|S|}\le \f{e_V}{|V|}\le \f{e_X}{|X|}=\rho(F),\]
	giving the result.
\end{proof}
As an aside, it could be the case that $\rho(T)\le \rho(F)$ whenever there exists a local isomorphism $\phi:V(T)\to V(F)$ (regardless of whether $T$ is a forest or not).  However, in this more general setting Lemma~\ref{lem:tech} no longer holds, as can easily be seen by considering a local isomorphism from $C_8$ to $C_4$.

\section{Concluding Remarks}\label{sec:con}
\textbf{Tur\'an problems in random graphs.} Given a graph $F$, we define its \textit{2-density}
\[m_2(F)=\max\left\{\f{e(F')-1}{v(F')-2}:F'\sub F,\ e(F')\ge 2\right\}.\]
Given a family of graphs $\c{F}$, we define $m_2(\c{F})=\min_{F\in \c{F}} m_2(F)$.
A simple deletion argument gives the following extension of Theorem~\ref{thm:main}.  Here the notation $f(n)\gg g(n)$ means $f(n)/g(n)$ tends to infinity as $n$ tends towards infinity.
\begin{cor}\label{cor:deletion}
	Let $(T,R)$ be a rooted tree with $\rho(T)\ge \f{b}{a}$ for $a,b$ positive integers. Then there exists some $\ell_0$ depending only on $(T,R)$ such that, for all $\ell\ge \ell_0$, a.a.s.
	\[\ex(G_{n,p},\c{T}^\ell)=\begin{cases}
		\Omega\left(\max\left\{p^{1-\f{a}{b}}n^{2-\f{a}{b}},n^{2-\rec{m_2(\c{T}^\ell)}}\right\}\right) & p\ge n^{\frac{-1}{m_2(\c{T}^\ell)}},\\ 
		(1+o(1))p{n\choose 2} & n^{\frac{-1}{m_2(\c{T}^\ell)}}\gg p\gg n^{-2} .
	\end{cases}\]
\end{cor}
\begin{proof}

	The bound $\Om(p^{1-\f{a}{b}}n^{2-\f{a}{b}})$ follows from Theorem~\ref{thm:main}. Trivially $\ex(G_{n,p},\c{T}^\ell)\le e(G_{n,p})$, and this is at most $(1+o(1))p{n\choose 2}$ a.a.s.\ provided $p\gg n^{-2}$ by the Chernoff bound.

	For each $F\in \c{T}^\ell$, let $F'\sub F$ be such that $m_2(F)=\f{e(F')-1}{v(F')-2}$, and let $X_F$ be the number of copies of $F'$ in $G_{n,p}$.  Note that $\E[X_F]=\Theta(p^{e(F')} n^{v(F')})$, and if $p\le  c n^{-1/m_2(\c{T}^\ell)}$ for some $c>0$ sufficiently small, then\footnote{Roughly speaking, this occurs provided $p^{e(F')} n^{v(F')}\le p n^2$, so it suffices to have $p\le n^{\f{v(F')-2}{e(F')-1}}=n^{-1/m_2(F)}$, and this follows from $p\le n^{-1/m_2(\c{T}^\ell)}$.} $\E[X_F]\le \frac{1}{2 |\c{T}^\ell|} \E[e(G_{n,p})]$ for all $F$.  Thus by considering the graph obtained by deleting an edge from each copy of $F'$ counted by each $X_F$, we find $\E[\ex(G_{n,p},\c{T}^\ell)]=\Omega(p n^2)$ for $p\le c n^{2-1/m_2(\c{T}^\ell)}$.  By monotonicity, this implies $\E[\ex(G_{n,p},\c{T}^\ell)]=\Om(n^{2-\rec{m_2(\c{T}^\ell)}})$ for $p\ge  n^{-1/m_2(\c{T}^\ell)}$.  That this bound holds a.a.s.\ for $p\ge  n^{-1/m_2(\c{T}^\ell)}$ follows easily from Azuma's inequality\footnote{This uses that the function $\ex(G_{n,p},\c{T}^\ell)$ changes by at most 1 whenever a single edge of $G_{n,p}$ is removed or added, and that $\E[\ex(G_{n,p},\c{T}^\ell)]\ge n^{2-1/m_2(\c{T}^\ell)}\gg \sqrt{n\choose 2}$.  This last inequality follows from $m_2(F)>1$ for all $F\in \c{T}^\ell$, which follows from every such $F$ containing a cycle.}, proving the lower bound for this range.
	
	For $p\ll n^{-1/m_2(\c{T}^\ell)}$ we use a somewhat more careful argument based off of \cite{MY}.  Define $F'$ and $X_F$ as above.   Let $\om_F=\f{\E[e(G_{n,p})]}{\E[X_F]}$, and by similar reasoning as above we have $\om_F\to \infty$.  Choose $\ep=\ep(n)$ any function such that $\ep\to 0$ and $\ep \om_F\to \infty$ for all $F$.  By Markov's inequality,
	\[\Pr\left[X_F\ge \ep p {n\choose 2}\right]\le \rec{\ep \om_F}\to 0.\]
	Thus a.a.s.\ $X_F=o(p{n\choose 2})=o(\E[e(G_{n,p})])$.  By the Chernoff bound, $e(G_{n,p})\ge (1-o(1))p{n\choose 2}$ a.a.s.\ for $p\gg n^{-2}$.  Thus by deleting an edge from each copy counted by some $X_F$, we obtain a $\c{T}^\ell$-free subgraph with $(1-o(1))p{n\choose 2}$ edges a.a.s.
\end{proof}

We say a rooted tree $(T,R)$ is \textit{balanced} if $\rho(T)=\f{e(T)}{v(T)-|R|}$.  It is easy to show  $\ex(n,\c{T}^\ell)=O(n^{2-1/\rho(T)})$ whenever $(T,R)$ is balanced (see \cite{bukh2018rational}), and in particular the lower bound of Corollary~\ref{cor:deletion} is tight at $p=1$ for balanced trees.  It is natural to ask if this lower bound is tight for all $p$, possibly up to some extra logarithmic factors.

\begin{quest}\label{quest:set}
	Which balanced trees $T$ are such that for all $\ell$, there exists a constant $c=c(\ell)$ such that for all $p$ we have a.a.s.
	\[\ex(G_{n,p},\c{T}^\ell)=O_\ell\left(\max\left\{p^{1-\rec{\rho(T)}}n^{2-\rec{\rho(T)}},n^{2-\rec{m_2(\c{T}^\ell)}}(\log n)^c\right\}\right).\]
\end{quest}

In fact, this upper bound might hold even if one only forbids a single element of $\c{T}^\ell$.  To this end, given a rooted graph $(F,R)$, we let $F^\ell$ denote the element of $\c{F}^\ell$ with the maximum number of vertices, i.e.\ the graph consisting of $\ell$ copies of $F$ which agree only at vertices of $R$.
\begin{quest}\label{quest:single}
	Which balanced trees $T$ are such that for all $\ell$, there exists a constant $c=c(\ell)$ such that for all $p$ we have a.a.s.
	\[\ex(G_{n,p},T^\ell)=O_\ell\left(\max\left\{p^{1-\rec{\rho(T)}}n^{2-\rec{\rho(T)}},n^{2-\rec{m_2(T^\ell)}}(\log n)^c\right\}\right).\]
\end{quest}
Question~\ref{quest:single} can be viewed as a significant generalization of a conjecture of Bukh and Conlon~\cite{bukh2018rational}, which asserts that $\ex(n,T^\ell)=\Theta(n^{2-1/\rho(T)})$ if $T$ is balanced and $\ell$ is sufficiently large.  This conjecture of Bukh and Conlon is unsolved despite receiving a massive amount of attention, and as such it seems hopeless to answer Question~\ref{quest:single} in general.  However, it is plausible that Question~\ref{quest:single} could be solved for trees $T$ which are known to satisfy Bukh and Conlon's conjecture.  There are many such $T$, see for example \cite{conlon2022rational,conlon2021more,janzer2020extremal,jiang2020negligible,jiang2023many,kang2021rational}.

The only positive answers to Questions~\ref{quest:set} and \ref{quest:single} that we are aware of come from work of Morris and Saxton~\cite{morris2016number} who gave a positive answer to Question~\ref{quest:single} when $T$ is a star and $R$ is its set of leaves.  They also gave a positive answer to Question~\ref{quest:single} when $T$ is a path and $R$ is its set of leaves, but they did this only for $\ell=2$.  In forthcoming work with McKinley, we extend this result of Morris and Saxton to give a positive answer to Question~\ref{quest:single} for all $\ell$ when $T$ is a path.  

The proofs of the results mentioned above all rely on the method of hypergraph containers together with a form of ``balanced supersaturation'' (which roughly say that if $G$ has $e(G)\gg n^{2-1/\rho(T)}$, then $G$ has about as many copies of $T^\ell$ as one would expect in a random graph, and that these copies are ``spread out'' in $G$).  It is possible that some of the current proofs showing $\ex(n,T^\ell)=O(n^{2-1/\rho(T)})$ for various graphs $T$ can be extended to give supersaturation, but we suspect that balanced supersaturation will be difficult. Balanced supersaturation might be easier to prove for the family $\c{T}^\ell$ since here supersaturation is very easy to prove, but we are unaware of how to show that these copies are sufficiently spread out.

In addition to the difficulties mentioned above, extra complications may arise whenever we do not have $m_2(T^\ell)=\f{e(T^\ell)-1}{v(T^\ell)-2}$.  This can happen even for relatively simple $T$.  For example, let $T$ be the tree obtained by subdividing an edge of $K_{1,3}$, and let $R$ be its set of leaves. Note that $(T,R)$ is a balanced tree.  Because $T^\ell$ contains $K_{2,\ell}$ as a subgraph, we have for $\ell\ge 2$ that \[m_2(T^\ell)\ge \f{e(K_{2,\ell})-1}{v(K_{2,\ell})-2}=\f{2\ell-1}{\ell}> \f{4\ell-1}{2\ell+1}=\f{e(T^\ell)-1}{v(T^\ell)-2}.\]
Perhaps the reason this phenomenon occurs is because $(T,R)$ is balanced but not ``strictly balanced'' in the sense that $\rho(T)=\f{e_S}{|S|}$ for some $S\subsetneq V(T)\sm R$, and it is possible that the problem is easiest for trees which do not exhibit this behavior.

\textbf{Generalized Tur\'an problems in random graphs.} Because Question~\ref{quest:single} is unlikely to be solved in general, it might be more tractable to look at generalized Tur\'an problems in random graphs for trees $T$ where positive answers to Questions~\ref{quest:set} or \ref{quest:single} are known.  For example, we ask the following.

\begin{quest}
	What is the behavior of $\ex(G_{n,p},K_{a,b},K_{s,t})$ when $t$ is sufficiently large in terms of $a,b,s$?
\end{quest}
Note that Theorem~\ref{thm:mainGeneralized}, together with the fact that $\rho(K_{1,s})=s$ when $R$ is the set of leaves, gives $\ex(G_{n,p},K_{a,b},K_{s,t})=\Om(p^{ab-ab/s}n^{a+b-ab/s})$ a.a.s., though a better bound for small $p$ can be obtained by a deletion argument.  When $p=1$ this bound was shown to be tight by Ma, Yuan, and Zhang~\cite{ma2018some} for all $a,b,s$.  Morris and Saxton~\cite{morris2016number} gave a (relatively) easy argument showing that when $a=b=1$, the lower bound from Theorem~\ref{thm:mainGeneralized} together with a deletion argument establishes the correct bounds for $\ex(G_{n,p},K_{a,b},K_{s,t})$.  It is plausible that an adaptation of these approaches could be used to give tight results for all $a,b,s,p$. 

\textbf{Hypergraphs.}  For ease of presentation we only stated our results in terms of graphs, but everything we did easily extends to $r$-uniform hypergraphs.  We briefly discuss how this works.

The notions of rooted hypergraphs, densities, local isomorphisms, and so forth, are defined exactly analogous to how these are defined in the setting of graphs.  Here the hypothesis of Proposition~\ref{prop:relTuran} must change to have $m\gg \Del_i^{1/(r-i)}$ for all $i$, where $\Del_i$ is the maximum $i$-degree of the $r$-graph $G$, and the current proof can easily be adapted to this new setting using an approach similar to that of \cite[Proposition 2.1]{SV-Cycles}.  The statement and proof of Proposition~\ref{prop:randPolyGen} is essentially word for word the same as it is currently written.  

The only minor complication arises with Proposition~\ref{prop:treeDensity}, and in particular what it means for a hypergraph to be a tree.  For our current proof to go through, we define an $r$-uniform forest to be an $r$-graph such that one can order its edges $e_1,\ldots,e_m$ in such a way that $|e_i\cap \bigcup_{j<i} e_j|\le 1$ for all $i$.  Given this, one can prove an analog of Lemma~\ref{lem:tech} by letting $u_1,\ldots,u_{r-1}\in e_m$ be the vertices of degree 1 of some minimal counterexample $T$, then arguing that $\phi:V(T-\{u_1,\ldots,u_{r-1}\})\to V(F)$ is still a local isomorphism, and then arguing more or less as the proof is currently written.

In total this approach will show  $\ex(G_{n,p}^r,H,\c{T}^\ell)=\Om(n^{v-\f{a}{b} e})$ when $T$ is an $r$-uniform tree and $\ell$ is sufficiently large.  Unfortunately this bound is very weak in general.  Indeed, if $T$ is a loose path with three edges and $R$ is a set of two leaves, then this bound gives $\ex(n,\c{T}^\ell)=\Om(n^{r-\f{3r-4}{3}})=\Om(n^{4/3})$.  However, it is easy to show $\ex(n,\c{T}^\ell)\ge \ex(n,T)=\Om(n^{r-1})$ by simply taking every edge containing a given vertex.  To give a lower bound which is actually effective, one should instead prove
\[\ex(G_{n,p}^r,H,\c{L}(\c{T}^\ell))=\Om(n^{v-\f{a}{b} e}),\]
where now we forbid all of $\c{L}(\c{T}^\ell)$ instead of just $\c{T}^\ell$.
Such a bound can easily be shown using the current argument after noting that the $G'$ obtained from Proposition~\ref{prop:relTuran} is in fact $\c{L}(\c{F})$-free (this follows from $\c{L}(\c{L}(\c{F}))=\c{L}(\c{F})$, which follows from the fact that the composition of local isomorphisms is again a local isomorphism).  If $T$ is a loose path on three edges, then $\c{L}(\c{T}^\ell)$ contains the set of Berge theta graphs with paths of length 3, and from this one can conclude $\ex(n,\c{L}(\c{T}^\ell))=O(n^{4/3})$, matching the lower bound given by this approach.  In particular, this approach recovers the lower bounds for the Tur\'an number of berge theta graphs obtained by He and Tait~\cite{he2019hypergraphs}.

\section*{Acknowledgments}
We thank the two referees for their valuable comments.

\bibliographystyle{abbrv}
\bibliography{Polynomial}
\end{document}